\documentclass[11pt,reqno]{amsart}
\usepackage{amsfonts, url}
\usepackage{amsmath,amssymb}

\def\marg#1#2{\def\marginnotetextwidth{\the\textwidth}\marginnote{\bf #1}{\bf #2}}

\numberwithin{equation}{section}

\newtheorem{theorem}{Theorem}[section]
\newtheorem{theorem*}{Theorem}
\newtheorem{prop}[theorem]{Proposition}
\newtheorem{lemma}[theorem]{Lemma}
\newtheorem{cor}[theorem]{Corollary}

\theoremstyle{definition}

\newtheorem{question}[theorem]{Question}
\newtheorem{example}[theorem]{Example}
\newtheorem{definition}[theorem]{Definition}
\newtheorem*{definition*}{Definition}
\newtheorem{remark}[theorem]{Remark}

\newcommand{\ch}{{\operatorname{char}}}

\newcommand{\Aut}{{\operatorname{ Aut}}}
\newcommand{\Inner}{{\operatorname{ Inn}}}
\newcommand{\Out}{{\operatorname{ Out}}}
\newcommand{\AIAut}{{\operatorname{ AIAut}}}
\newcommand{\AID}{{\operatorname{ AID}}}
\newcommand{\Der}{{\operatorname{ Der}}}
\newcommand{\ad}{{\operatorname{ ad}}}
\newcommand{\Autc}{{\operatorname{ Aut_{\text{\rm{c}}}}}}

\newcommand{\Lie}{{\operatorname{ Lie}}}

\newcommand{\diag}{{\operatorname{ diag}}}

\newcommand{\Sym}{{\operatorname{ Sym}}}

\newcommand{\FSym}{{\operatorname{ FSym}}}

\newcommand{\End}{{\operatorname{ End}}}

\newcommand{\Ham}{{\operatorname{ Ham}}}

\newcommand{\Span}{{\operatorname{Span}}}

\newcommand{\ph}{\varphi}

\font\cyr=wncyr10 scaled \magstep1%

\def\Sha{\text{\cyr Sh}}

\def\g{\mathfrak g}

\begin{document}

\title{Tate--Shafarevich groups and algebras}

\author{Boris Kunyavski\u\i , Vadim Z. Ostapenko}




\address{Kunyavski\u\i : Department of
Mathematics, Bar-Ilan University, 5290002 Ramat Gan, ISRAEL}
\email{kunyav@macs.biu.ac.il}

\address{Ostapenko:  Department of
Mathematics, Bar-Ilan University, 5290002 Ramat Gan, ISRAEL}
\email{ostap@math.biu.ac.il}

\begin{abstract}
The Tate--Shafarevich set of a group $G$ defined by Takashi Ono coincides, in the case where
$G$ is finite, with the group of outer class-preserving automorphisms of $G$ introduced by
Burnside. We consider analogues of this important group-theoretic object for Lie algebras
and associative algebras and establish some new structure properties thereof. We also discuss
open problems and eventual generalizations to other algebraic structures.
\end{abstract}

\thanks{This research was supported by the ISF grant 1994/20. A substantial part
of this work was done during the visit of the first named author to the MPIM (Bonn). The work
was finished during his visit to the IHES (Bures-sur-Yvette).
Support of these institutions is gratefully acknowledged.}

\maketitle

\section{Introduction} \label{sec:intro}

A starting point of this research is the following purely group-theoretic notion.
Let $G$ be an abstract group acting on itself by conjugation, and let $H^1(G,G)$
denote the first group cohomology corresponding to this action.

\begin{definition}
The (pointed) set
\begin{equation} \label{Sha-gr}
\Sha(G):=\ker \left[ H^1(G,G)\to \prod_{C<G \text{ cyclic }}H^1(C,G) \right] 
\end{equation}
is called the {\bf Tate--Shafarevich set} of $G$.
\end{definition}

The definition and the name were introduced by Takashi Ono \cite{On1}, \cite{On2}.
The local-global flavour justifies the allusion to the object bearing the
same name which appeared in the arithmetic-geometric context (related to the
action of the absolute Galois group of a number field $K$ on the group $A(\overline K)$
of $\overline K$-points of an abelian $K$-variety $A$).
Recall that the usage of the Cyrillic letter  $\Sha$ (``Sha'') in this notation was initiated by Cassels
because of its appearance as the first letter in the surname of Shafarevich.

Formula \eqref{Sha-gr} admits a more down-to-earth interpretation, attributed in \cite{On2}
to Marcin Mazur (note that it appeared implicitly in an earlier paper by Chih-Han Sah \cite{Sah}).
For the reader's convenience, we reproduce this argument.

Recall that 1-cocycles $Z^1(G,G)$ are crossed homomorphisms, i.e.
maps $\psi\colon G\to G$ with the property
$$
\psi (st) = \psi (s) {}^s\psi (t) = \psi (s) s \psi (t) s^{-1}.
$$
Then the correspondence $\psi(s)\mapsto f(s) = \psi(s) \cdot s$ gives a bijection between
$Z^1(G,G)$ and $\End(G)$.
Under this correspondence, 1-coboundaries correspond to inner automorphisms.
Further, a 1-cocycle whose cohomology class becomes trivial after restriction to every cyclic
subgroup corresponds to an almost inner (=locally inner=pointwise inner=class preserving) endomorphism,
i.e. $f\in\End(G)$ with the property $f(g)=a^{-1}ga$ (where $a$ depends on $g$).
Note that any class preserving endomorphism is injective. If $G$ is finite, it
is also surjective, and we arrive at the object introduced by Burnside \cite{Bur1} more than 100 years ago:
$$\Sha (G)\cong \AIAut(G)/\Inner(G),
$$
where $\AIAut(G)$ (sometimes denoted by $\Autc(G)$) stands for the group of almost inner automorphisms of $G$.
In particular, if $G$ is finite, $\Sha (G)$ is a group, not just a pointed set.
(Ono \cite{On2} extended this to the case where $G$ is profinite.)

There are many classes of groups $G$ with trivial $\Sha (G)$, see the surveys \cite{Ku1}, \cite{Ya}
where such groups are called $\Sha$-rigid. One can also find there some interesting examples
with nontrivial $\Sha (G)$. Such examples often give rise to counter-examples to some difficult problems,
such as Higman's problem on isomorphism of integral group rings.

Our first goal is to study the Lie-algebraic analogue of $\Sha (G)$, emphasizing its cohomological
nature and local-global flavour. This analogue, under different names, appeared in the literature.
In the pioneering work by Carolyn Gordon and Edward Wilson \cite{GW}, this object was studied
in the differential-geometric context, allowing them to produce a continuous family of isospectral non-isometric
compact Riemann manifolds. Recently, the interest to these Lie-algebraic structures was revived in the series of papers
by Farshid Saeedi and his collaborators \cite{SMSB}, \cite{SMS1}, \cite{SMS2}, and also in the series of papers by
Burde, Dekimpe and Verbecke \cite{BDV1}, \cite{BDV2}, \cite{BDV3}.

In Section \ref{sec:lie}, we consider this {\bf Tate--Shafarevich Lie algebra} $\Sha (\g)$ and its generalization
$\Sha (\g,M)$, where $\g$ is a finite-dimensional Lie algebra and $M$ is a $\g$-module. Our main contribution here
is the proof of the fact that $\Sha (\g)$ is an ideal in the Lie algebra $\Der (\g)$ of outer derivations of $\g$
in the case where $\g$ is nilpotent. This gives a partial answer to a question posed in \cite{BDV1}.

Our next aim is to extend the parallelism between groups and Lie algebras by considering associative algebras and
introducing analogous objects. For such an algebra $A$, in Section \ref{sec:ass}, we define two analogues of
$\Sha (G)$ and $\Sha (\g)$, additive $\Sha_{\text{\rm{a}}}(A)$ and multiplicative $\Sha_{\text{\rm{m}}}(A)$,
and study their simplest properties and relations to the objects defined earlier. We hope that
these new local-global invariants of associative algebras will prove useful, as their predecessors.

Finally, in Section \ref{sec:par}, we overview some open problems arising from the parallelism
among the objects of the triad consisting of Lie algebras, associative algebras, and groups.
We also speculate on extending this parallelism from Lie algebras to other algebraic structures,
such as Malcev algebras, Leibniz algebras, Poisson algebras, and the corresponding triads whenever
they exist.

\medskip

\noindent
{\it Notation and conventions.}
Unless stated otherwise, $k$ is an arbitrary field of
characteristic zero,
$\g$ is a finite-dimensional Lie $k$-algebra,
$A$ is an associative $k$-algebra, and
$M$ is either a $\g$-module, or an $A$-bimodule,
depending on the context.

\section{Lie algebras} \label{sec:lie}

\subsection{Preliminaries} \label{subsec:lie-prelim}
Let $\g$ be a Lie algebra over $k$, and let $M$ be a (left) $\g$-module, i.e. $M$ is a vector $k$-space
and there exists a $k$-bilinear map
$$
\g\times M \to M, \quad (g,m)\mapsto g\circ m,
$$
such that
$$
[g,h]\circ m = g\circ (h\circ m) - h\circ (g\circ m)
$$
for all $g,h\in \g, m\in M$. In particular, $M=\g$
is a $\g$-module with respect to the adjoint action $g\circ k=[g,k]$
because of the Jacobi identity.

Further, recall that a derivation $D\colon \g\to M$ is a $k$-linear map
such that
\begin{equation} \label{der-module-lie}
D([g,h])=g\circ D(h) - h\circ D(g)
\end{equation}
for all $g,h\in \g$. For a given $m\in M$, the map
$$
D_m\colon \g\to M, \quad g\mapsto g\circ m,
$$
is a derivation. Such derivations are called inner. We denote by
$\Der (\g,M)$ the set of all derivations and by $\ad (\g,M)$ the
set of all inner derivations. Clearly, they are both vector $k$-spaces,
and $\ad (\g ,M)$ is a $k$-subspace of $\Der (\g,M)$. Let $\Out (\g ,M)=\Der (\g,M)/\ad (\g,M)$
denote the quotient vector space. It is well known that $\Out (\g ,M)$ is the first
Chevalley--Eilenberg cohomology $H^1(\g,M)$.

In the special case
$M=\g$,  formula \eqref{der-module-lie} is the usual Leibniz rule.
We abbreviate the notation $\Der (\g,\g)$, $\ad (\g,\g)$ and $\Out (\g ,M)$ to
$\Der (\g)$, $\ad (\g)$ and $\Out (\g )$, respectively. The first two vector spaces acquire a natural Lie algebra
structure defined by the Lie bracket $[D,D']=DD'-D'D$. Therefore, since $\ad (\g)$ is a Lie ideal
of $\Der (\g )$, $\Out (\g )$ also carries a Lie algebra structure. It is
the Chevalley--Eilenberg cohomology $H^1(\g,\g)$ related to the adjoint action of $\g$.

\begin{definition} \label{def:lid-lie}
$$
\AID (\g,M):= \{D\in \Der (\g,M) \quad | \quad (\forall g\in \g) \quad  (\exists m\in M) \quad D(g)=g\circ m\}.
$$
(Here $m$ may depend on $g$.) We call elements of $\AID (\g,M)$ {\bf{almost inner derivations}}
of $\g$ with coefficients in $M$.
\end{definition}

\begin{remark}
Perhaps a more appropriate name for objects introduced in Definition \ref{def:lid-lie} would be locally inner derivations.
It would better reflect their local-global flavour. We have chosen another name, following \cite{GW} and \cite{BDV1},
in order to avoid notational collisions. First, locally inner derivations are used in the theory of Banach algebras (having a different
meaning). Second, this term is too close to local derivations, which are yet another object, intensely studied
over past years and having important applications. 
\end{remark}

Clearly, $\AID (\g,M)$ is a subspace of $\Der (\g ,M)$ and $\ad (\g,M)$ is a subspace of $\AID (\g,M)$,
so we define
$$
\Sha (\g,M):=\AID (\g,M)/\ad (\g ,M).
$$
This quotient is a subspace of $\Out (\g,M)$.

As above, we shorten $\AID (\g):=\AID (\g,\g)$, and so on.
These algebras were considered in \cite{GW}, \cite{SMSB}, \cite{BDV1}.
As mentioned in Section \ref{sec:intro},
algebras $\g$ with nonzero $\Sha (\g)$ exhibited in the aforementioned papers often reveal important
geometric phenomena, see \cite{GW} for details. Note that $\AID (\g )$ inherits the Lie algebra structure from
$\Der (\g)$ (see \cite[Proof of Proposition 2.3]{BDV1}), $\ad (\g )$ is a Lie ideal in $\AID (\g )$, and hence
$\Sha (\g )$ also carries a natural Lie algebra structure.

\begin{definition} \label{def:sha-lie}
We call $\Sha (\g)$ the {\bf {Tate--Shafarevich algebra}} of $\g$.
\end{definition}

\subsection{Properties}
We start with a basic structural question posed in \cite{BDV1}.

\begin{question} \label{q:ideal-lie}
Is $\AID(\g)$ an ideal of $\Der (\g)$?
\end{question}

If this question is answered in the affirmative, we conclude that $\Sha (\g)$
is an ideal of $\Out (\g)$.

So far, Question \ref{q:ideal-lie} is wide open. The next result can be viewed as a first step.

\begin{theorem} \label{th:ideal-lie}
Let $\g$ be a finite-dimensional nilpotent Lie algebra over $k=\mathbb C$.
Then  $\AID(\g)$ is an ideal of $\Der (\g)$, and hence $\Sha (\g)$
is an ideal of $\Out (\g)$.
\end{theorem}

\begin{proof}
{\it Step 1.}
Note that since $\g$ is nilpotent, it is algebraic, i.e. there exists
an affine algebraic $k$-group $G$ such that $\g=\Lie (G)$.

{\it Step 2.}
Denote $N:=\Aut (G)$, the group of automorphisms of $G$. It also has a structure of
an algebraic group, and we have an isomorphism of Lie algebras $\Lie (N)\cong \Der(\g)$,
see, e.g. \cite[Section~I.2.10]{OV} (the material of this section refers to Lie groups
but since $k=\mathbb C$, the same holds for algebraic groups).

{\it Step 3.} Denote $H:=\AIAut(G)$, the group of almost inner automorphisms of $G$,
see Section \ref{sec:intro}. Let us prove that $H$ is a closed normal subgroup of $N$.

We thank Pradeep Kumar Rai
for communicating us the following  fact. 

\begin{lemma} \label{lem:norm}
Let $G$ be any group. Then $\AIAut(G)$ is a normal subgroup of $\Aut(G)$.
\end{lemma}

\begin{proof}
Let $\sigma\in\AIAut(G)$, $\varphi\in\Aut (G)$, $g\in G$. By the definition of
an almost inner automorphism, we have $\sigma(\varphi (g))=a\varphi (g)a^{-1}$
for some $a\in G$. Hence
$$
(\varphi^{-1}\sigma\varphi)(g) = \varphi^{-1}(a\varphi (g)a^{-1})
= \varphi^{-1}(a) g \varphi^{-1}(a^{-1}) = \varphi^{-1}(a) g (\varphi^{-1}(a))^{-1}.
$$
Thus $\varphi^{-1}\sigma\varphi\in\AIAut (G)$.
\end{proof}

To prove that $H$ is closed in $N$, we use the same argument as in \cite{GW}. Again,
as at Step 2, we only have to rephrase it, replacing Lie groups with algebraic groups.

{\it Step 4.} By Step 3, $H$ is an affine algebraic group. Then $\Lie(H)$ is an
ideal of $\Der(L)$, see, e.g. \cite[10.2,~Cor.~A]{Hu}.

{\it Step 5.} As in \cite{GW}, we have an isomorphism $\Lie (H)\cong \AID(\g)$.
By Step 4, this finishes the proof.

\end{proof}

\begin{remark} \label{rem:Milne}
It is unclear whether one can extend the class of algebras for which the statement
of Theorem \ref{th:ideal-lie} holds. Already Step 1 of our proof breaks down for
solvable algebras because some of them are not algebraic, see
\cite[\S 5, Ex.~6 on p.~126]{Bou}, \cite[1.25, 3.42]{Mi}.
\end{remark}

\begin{example}
Consider the 5-dimensional solvable Lie algebra $\g$ mentioned in Remark \ref{rem:Milne}.
It is defined by the following multiplication table of basis elements:
$[e_1,e_2]=e_5,
[e_1,e_3]=e_3,
[e_2,e_4]=e_4$
(all other products are equal to $0$),
so $e_5$ is a central element.

We have $\g^{(n)}:=[\g,g^{(n-1)}]=\Span(e_3,e_4)$ for all $n\geq 2$.
On the other hand, $\g^{\prime\prime}:=[[\g,\g],[\g,\g]]=0$. Therefore, $\g$
is solvable but not nilpotent.

Let $\ph$ be a derivation of $\g$ with matrix $C=(\ph_{ji})_1^5$. A straightforward
computation of the $\ph (e_i)$ using Leibniz rule implies that all entries of $C$ are
zero except $\ph_{31}, \ph_{33}, \ph_{42}, \ph_{44}, \ph_{51}, \ph_{52}$, with no other
relations. Hence
\[
\Der(\g)=\Span(E_{31},E_{33},E_{42},E_{44},E_{51},E_{52}),
\]
where the $E_{ji}$ denote the matrix units.

Further, suppose that $\ph$ is an almost inner derivation so that the linear equation in $y\in \g$
\begin{equation} \label{eq:almost}
[x,y]=\ph(x)
\end{equation}
has a solution for every $x\in\g$. Representing $x$ and $y$ as vectors in $k^n$ in the
basis $\{e_i\}$ and using the conditions $\ph (e_i)\in [e_i,\g]$ and the multiplication table,
we present \eqref{eq:almost} as a system of linear equations in the coordinates of $y$.
We then use the condition that this system must be solvable for any choice of the
coordinates of $x$ and arrive at the following conclusion: the entries of $C$ must satisfy
the additional relations $\ph_{52}=\ph_{33}$, $\ph_{51}=-\ph_{44}$, and these relations are sufficient to
guarantee that $\ph$ is an almost inner derivation. Thus
\begin{equation} \label{eq:aid}
\AID(\g)=\Span(E_{31},E_{33}+E_{52}, E_{42}, E_{44}- E_{51}).
\end{equation}

Finally, suppose further that $\ph$ is an inner derivation. We then compute $\ph(e_i)$ using the
multiplication table and arrive at the same result as in \eqref{eq:aid}, i.e.
$\ad(\g)=\Span(E_{31},E_{33}+E_{52}, E_{42}, E_{44}- E_{51}).$ We conclude that $\Sha(\g)=\AID(\g)/\ad(\g)=0$.

Thus for the algebra we considered, $\Sha (\g)$ is an ideal of $\Out (\g)$ for trivial reasons.
It remains a tempting problem to find an example where Question \ref{q:ideal-lie} is answered
in the negative.
\end{example}

\begin{remark} \label{nonzero}
Note that there are numerous examples of Lie algebras $\g$ with $\Sha (\g)\ne 0$, see
\cite{GW}, \cite{SMSB}, \cite{BDV1}--\cite{BDV3}. To give the reader some flavour, we
reproduce here some of such examples. We use the notation of Magnin's tables \cite{Ma}
for algebras of small dimension.

Let $\g =\g_{5,3}$ be given by a basis $e_1,\dots ,e_5$ with multiplication table
$$
[e_1,e_2] = e_4, \, [e_1,e_4] = e_5, [e_2,e_3] = e_5
$$
(all other entries equal 0). This is a 3-step nilpotent algebra with $\AID (\g)=\ad (\g) \oplus
\left<E_{5,3}\right>$ where $E_{5,3}$ maps $e_3$ to $e_5$ and all other $e_i$ to 0. It is shown in
\cite[Example 2.7]{BDV1} that the derivation $E_{5,3}$ is not inner, so that $\Sha (\g)$ is
an abelian Lie algebra of dimension 1.

We are grateful to the referee for pointing out other interesting examples of algebras $\g$ with
$\Sha (\g)\ne 0$. These are $\g =\g_{6,20}$, the nilradical of the
standard Borel subalgebra of the exceptional simple Lie algebra $\g_2$ (see \cite[Remark~8.6]{BDV1})
and the filiform Witt algebra $W_n$ for $n\ge 9$ (see \cite[Proposition~4.5]{BDV2}).

It is also worth noting that such examples cannot be found among algebras of dimension less than 5:
according to \cite[Proposition~2.8]{BDV1}, every almost inner derivation of such an algebra is inner.
\end{remark}

We shall discuss more (known and unknown) properties of $\Sha (\g)$ in Section \ref{sec:par},
in the context of parallelism among groups, Lie algebras, and associative algebras.

\medskip

Meanwhile, for the sake of application in the next section, we consider the algebra
$\Sha (\g, M)$ in the special case
$M=U(\g)$, the universal enveloping algebra of $\g$,
where the module structure is given by the adjoint action of $\g$
continuing the adjoint action of $\g$ on itself. (There are other actions
that we do not consider here.) Recall that the Poincar\'e--Birkhoff--Witt
(PBW) theorem provides the canonical $k$-linear injective map $i\colon \g\to U(\g )$.
We will identify $\g$ with its image $i(\g )$ without special mentioning.

We start with the following simple (and perhaps well-known) lemma.

\begin{lemma} \label{lem:adj}
Let $m\in U(\g )$. If for all $g\in \g$ we have $g\circ m\in \g$, then $m\in \g$.
\end{lemma}

\begin{proof}
Recall that $U(\g )$ has a natural filtration
$$
U_0\subset U_1\subset\dots\subset U_m\subset\dots
$$
where $U_i$ is spanned by the monomials of length at most $i$.

The associated grading gives, by canonical symmetrization, a decomposition
\begin{equation} \label{sum}
U(\g )=\bigoplus_{m\ge 0}U^m
\end{equation}
where $U^0=U_0=k$, and $U^m=U_m/U_{m-1}$ ($m\ge 1)$ is the set of symmetric homogeneous
elements of degree $m$ (in particular, $U^1(\g )$ is isomorphic to $\g$, and each direct summand
is $\g$-invariant. See, e.g. \cite[2.4.6, 2.4.10]{Dix} for details.

This immediately implies the assertion of the lemma.
\end{proof}

We will need the following proposition where some of the statements are well
known.

\begin{prop} \label{prop:emb}
The map $i$ induces $k$-linear injective maps
\begin{itemize}
\item[(i)] $\Der(\g)\to \Der(\g,U(\g))$;
\item[(ii)] $\ad(\g)\to \ad(\g,U(\g))$;
\item[(iii)] $\Out(\g)\to \Out(\g,U(\g))$;
\item[(iv)] $\AID(\g)\to \AID(\g,U(\g))$;
\item[(i)] $\Sha(\g)\to \Sha(\g,U(\g))$.
\end{itemize}
\end{prop}

\begin{proof}

Assertions (i), (ii) and (iv) are obvious, (iii) and (v)
are immediate consequences of Lemma \ref{lem:adj}.

\end{proof}

\begin{cor} \label{nonzero-Sha-U}
There exist finite-dimensional Lie algebras $\g$ with nonzero $\Sha (\g ,U(\g ))$.
\end{cor}

\begin{proof}
By Proposition \ref{prop:emb}(v), for any Lie algebra with $\Sha (\g)\ne 0$ we have
$\Sha (\g, U(\g ))\ne 0$. There are many examples of such algebras, see Remark
\ref{nonzero} above.
\end{proof}

\section{Associative algebras} \label{sec:ass}

In this section, $k$ is a field, $A$ is an associative unital $k$-algebra,
and $M$ is an $A$-bimodule. We do not use any special symbols
for denoting multiplication in $A$ and left and right actions
of $A$ on $M$ with the hope that this does not lead to any confusion.

In this case we have two versions of $\Sha (A)$, additive and
multiplicative.

\subsection{Additive $\Sha (A)$} \label{subsec:add}

Recall that a derivation $D\colon A\to M$ is a $k$-linear map
such that
$$
D(ab)=D(a)b+aD(b)
$$
for all $a,b\in A$. For a given $m\in M$, the map
$$
D_m\colon A\to M, \quad m\mapsto am-ma,
$$
is a derivation. Such derivations are called inner. We denote by
$\Der (A,M)$ the set of all derivations and by $\ad (A,M)$ the
set of all inner derivations. Clearly, they are both vector $k$-spaces,
and $\ad (A,M)$ is a $k$-subspace of $\Der (A,M)$. Let $\Out (A,M)=\Der (A,M)/\ad (A,M)$
denote the quotient space. It is well known that $\Out (A,M)$ is the first
Hochschild cohomology $HH^1(A,M)$.

In the special case
$M=A$ we abbreviate the notation $\Der (A,A)$, $\ad (A,A)$ and $\Out (A,A)$ to
$\Der (A)$, $\ad (A)$ and $\Out (A)$, respectively. The first two spaces acquire a natural Lie algebra
structure defined by the Lie bracket $[D,D']=DD'-D'D$, $\ad (A)$ is a Lie ideal
of $\Der (A)$, hence $\Out (A)$ also carries a Lie algebra structure. This Lie algebra is
the first Hochschild cohomology $HH^1(A)$.

\begin{definition} \label{def:lid-ass}
Set
$$
\AID (A,M):= \{D\in \Der (A,M) \quad | \quad (\forall a\in A) \quad  (\exists m\in M) \quad D(a)=am-ma\}.
$$
(Here $m$ may depend on $a$.) We call elements of $\AID (A,M)$ almost inner derivations
of $A$ with coefficients in $M$.
\end{definition}

Clearly, $\AID (A,M)$ is a subspace of $\Der (A,M)$, $\ad (A,M)$ is a subspace of $\AID (A,M)$,
and we define
$$
\Sha_{\text{\rm{a}}} (A,M):=\AID (A,M)/\ad (A,M).
$$
It is a subspace of $\Out (A,M)$.

As in Section \ref{sec:lie}, in the particular case $M=A$ we shorten $\AID (A,M)$ and $\Sha_{\text{\rm{a}}} (A,M)$ to
$\AID (A)$ and $\Sha_{\text{\rm{a}}} (A)$, respectively. As above, $\AID(A)$ inherits the Lie algebra structure from
$\Der (A)$.

Indeed, the same argument as in \cite[Proof of Proposition 2.3]{BDV1} works here as well.

\begin{prop}
For any $D,D'\in \AID (A)$ we have $[D,D']\in \AID (A)$.
\end{prop}
\begin{proof}
Let $a\in A$.
We have $D(a)=am-ma$, $D'(a)=am'-m'a$ for some $m,m'\in A$ depending on $a$, so that
{\footnotesize{
$$
\begin{aligned}
& [D,D'](a)  =(DD'-D'D)(a)=D(D'(a))-D'(D(a))=D(am'-m'a)-D'(am-ma)\\
          & = D(a)m'+aD(m')-D(m')a-m'D(a)-D'(a)m-aD'(m)+D'(m)a+mD'(a)\\
          & =(am-ma)m'+aD(m')-D(m')a-m'(am-ma)-(am'-m'a)m-aD'(m)+D'(m)a+m(am'-m'a)\\
          & =amm'-mam'+aD(m')-D(m')a-m'am+m'ma-am'm+m'am-aD'(m)+D'(m)a+mam'-mm'a \\
          & =an-na,
\end{aligned}
$$
}}
where $n=mm'-m'm-D'(m)+D(m')$. Hence $[D,D']\in \AID (A)$.
\end{proof}

Clearly,
$\ad (A)$ is a Lie ideal in $\AID (A)$, and hence
$\Sha_{\text{\rm{a}}}(A)$ also carries a natural Lie algebra structure.

\begin{definition} \label{def:sha-add}
We call $\Sha_{\text{\rm{a}}}(A)$ the {\bf {additive Tate--Shafarevich algebra}} of $A$.
\end{definition}

Once a new object is introduced, the first question to ask is whether it can be nontrivial.
It is not hard to construct an associative algebra $A$ with nonzero $\Sha_{\text{\rm{a}}}(A)$.
Here is a `generic' construction suggested by Leonid Makar-Limanov (a similar construction was
communicated to us by Alexei Kanel-Belov; cf. also Example \ref{ex:mult} below).

\begin{example} \label{ex:add}
Take a non-commutative algebra $A$ with an infinite set $S$ of generators and finitary multiplication table,
i.e. such that only a finite number of generators do not commute with any given generator.
Let $m$ denote a formal infinite sum of elements of $A$ such that every generator appears only in a finite
number of summands of $m$. Then the map
$$
D_m\colon A\to A, \quad a\mapsto am-ma,
$$
is well-defined and is a derivation of $A$. Clearly, this derivation is almost inner but not inner,
so that $\Sha_{\text{\rm{a}}}(A)\ne 0.$
\end{example}

Our further goal is to exhibit a {\it finitely generated} algebra $A$ with $\Sha_{\text{\rm{a}}}(A)\ne 0$.
Towards this end, consider $A=U(\g )$ where $\g$ is a Lie algebra, and $U(\g )$ is its universal enveloping algebra.
Any $\g$-bimodule $M$ has a unique structure of a $U(\g)$-bimodule.

\begin{lemma} \label{lem:sha-add-lie}
\begin{itemize}
\item[]
\item[(i)] For any $\g$-bimodule $M$ the vector $k$-spaces
$\Sha_{\text{\rm{a}}}(U(\g), M)$ and $\Sha (\g,M)$ are isomorphic.
\item[(ii)]
The Lie algebras $\Sha_{\text{\rm{a}}}(U(\g))$ and $\Sha (\g,U(\g ))$ are isomorphic.
\end{itemize}\end{lemma}

\begin{proof}
First recall that every derivation $D\colon \g\to M$ can be uniquely extended
to a derivation $D'\colon U(\g )\to M$ (see, e.g. \cite[Lemma~2.1.3]{Dix} or
\cite[XIII.2]{CE}). (One has to continue $D$ using the canonical embedding
$\g\to U(\g)$ and then use Leibniz rule.) Under this process, the inner
derivations $\ad (\g, M)$ go to the inner derivations $\ad (U(\g), M)$,
and $\AID (\g, M)$ goes to $\AID (U(\g), M)$. This proves (i). It is easy to see that
in the case $M=U(\g )$ the Lie bracket $[D_1,D_2]$ of derivations of $\g$ goes to
$D'_1D'_2-D'_2D'_1$ where $D'_i$ $(i=1,2)$ are the corresponding derivations of $U(\g)$.
This proves (ii).
\end{proof}

\begin{cor}
There exist finitely generated associative algebras $A$ with $\Sha_{\text{\rm{a}}}(A)\ne 0$.
\end{cor}

\begin{proof}
Let $A=U(\g )$ where $\g$ is a finite-dimensional Lie algebra with nonzero
$\Sha (\g ,U(\g))$. In view of explicit examples mentioned in Remark \ref{nonzero},
such Lie algebras exist, see Corollary \ref{nonzero-Sha-U}.
By Lemma \ref{lem:sha-add-lie}, we have $\Sha_{\text{\rm{a}}}(A)\ne 0$.
\end{proof}

\medskip

The algebra $U(\g)$ is infinite-dimensional, so the next step is to look for
{\it {finite-dimensional}} associative algebras $A$ with $\Sha_{\text{\rm{a}}}(A)\ne 0$.
Somewhat degenerate examples arise from the following observation (see, e.g.
\cite[Proposition~1]{GR}): a Lie algebra $\g$ is associative if and only if it is
two-step nilpotent. As examples of two-step nilpotent Lie algebras $\g$ with
$\Sha(\g)\ne 0$ can be produced in abundance, see \cite{BDV2},
we obtained the needed associative algebras $A$ for free. Note, however, that
the obtained associative algebras are obviously not unital. To repair this,
one can use a standard procedure of adjoining the unit to get a unital algebra
$\widetilde A:=k\oplus A$ for which we have
$\Sha_{\text{\rm{a}}}(\widetilde A)=\Sha_{\text{\rm{a}}}(A)\ne 0$.

It is tempting to use the same examples of finite-dimensional nilpotent Lie algebras $\g$ with nonzero $\Sha (\g)$
to construct `genuine' examples of finite-dimensional associative algebras $A$ with nonzero
$\Sha_{\text{\rm{a}}}(A)$.

Let us first record some obvious properties of derivations in the following lemma the proof of which
is straightforward. Let $A$ be an associative algebra, and let $\g=\Lie (A)$ be its Lie algebra
(the underlying vector $k$-space of $A$ equipped with the bracket $[x,y]=xy-yx$).

\begin{lemma} \label{ass-lie}
\begin{itemize}
\item[]
\item[(i)] $\Der(A)\subseteq\Der (\g)$.
\item[(ii)] $\ad(A) = \ad(\g)$.
\item[(iii)] $\AID (A) =\AID (\g)\cap \Der (A)$.
\item[(iv)] $\Sha_{\text{\rm{a}}}(A)\subseteq\Sha (\g)$.
\end{itemize}
\end{lemma}

This implies that genuine examples we are looking for cannot be too small:

\begin{cor}
If $\dim (A)\le 4$, then $\Sha_{\text{\rm{a}}}(A)=0$.
\end{cor}

\begin{proof}
By \cite[Proposition~2.8]{BDV1}, for $\g =\Lie(A)$ we have $\Sha (\g)=0$
(see Remark \ref{nonzero} above).
The assertion now follows from Lemma \ref{ass-lie}(iv).
\end{proof}



\subsection{Multiplicative $\Sha (A)$} \label{subsec:mult}

Let $G=\Aut_k(A)$ be the group of all $k$-algebra automorphisms of $A$. In the sequel,
we shorten $\Aut_k(A)$ to $\Aut(A)$. Let $A^{\times}$ denote
the group of invertible elements of $A$. Denote by
$\Inner (A)$ the group of inner automorphisms of $A$. Recall that
$\varphi\in \Inner (A)$ if there exists $a\in A^{\times}$ such that
$\varphi(x)=axa^{-1}$. $\Inner (A)$ is a normal subgroup of $\Aut (A)$.

\begin{definition}
Define
$$
\AIAut(A):=\{\varphi\in\Aut (A) \quad | \quad (\forall x\in A) \quad  (\exists a\in A^{\times}) \quad \varphi(x)=axa^{-1}\}.
$$
(Here $a$ may depend on $x$.) We call elements of $\AIAut (A)$ {\bf{almost inner automorphisms}} of $A$.
\end{definition}

Clearly, $\Inner (A)$ is a normal subgroup of $\AIAut (A)$.

\begin{definition}
The group $$\Sha_{\text{\rm{m}}}(A):=\AIAut(A)/\Inner (A)$$
is called the {\bf {multiplicative Tate--Shafarevich group}} of $A$.
\end{definition}

As in Section \ref{subsec:add}, we first make sure that there exist $A$ with
$\Sha_{\text{\rm{m}}}(A)\ne 0$. The following example (provided by Be'eri
Greenfeld) is parallel to Example \ref{ex:add}.

\begin{example} \label{ex:mult}
Let $A$ be the algebra of (countably) infinite matrices $S$ over
$k$ which are eventually scalar (namely, for $i+j\gg 1, S(i,j)=\lambda \delta_{i,j}$ for some $\lambda\in k$).
Consider the automorphism of $A$ induced by conjugation by an infinite diagonal matrix
$\diag(\lambda_1,\lambda_2,\dots)$ with distinct nonzero $\lambda_i$'s.
This is an almost inner automorphism of $A$ which is not inner. Hence $\Sha_{\text{\rm{m}}}(A)\ne 0$.
\end{example}

\begin{remark}
Both Examples \ref{ex:add} and \ref{ex:mult} are reminiscent of a similar well-known construction arising in the group-theoretic
set-up. Namely, let $G=\FSym(\Omega )$ be a finitary symmetric group (the group of all permutations of an infinite set $\Omega$
fixing all but finitely many elements of $\Omega$). Viewing $G$ as a subgroup of the symmetric group $\Sym (\Omega)$, consider
an automorphism $\varphi\colon G\to G$ induced by conjugation by some $a\in \Sym (\Omega)\setminus \FSym (\Omega)$. Clearly,
$\varphi$ is almost inner but not inner. Actually, in this case $\AIAut(G)/\Inner(G)$ is isomorphic to the infinite simple group
$\FSym(\Omega)/\Sym (\Omega)$ (this observation is attributed to Passman, see \cite[Introduction]{Sah}), and  $\Sha(G)$ is even
larger because there are non-surjective almost inner endomorphisms \cite{AE}.
\end{remark}

As in Section \ref{subsec:add}, we are interested in exhibiting examples of finitely generated (or even finite-dimensional) algebras $A$
with nontrivial $\Sha_{\text{\rm{m}}}(A)$.

In the case $A=U(\g)$, considered in Section \ref{subsec:add} in the context of the additive $\Sha$,
we did not succeed in presenting an example of $\g$ with $\Sha_{\text{\rm{m}}}(U(\g))\ne 0$.

Consider finite-dimensional algebras $A$. In this case, $G$ can be equipped
with a structure of an affine algebraic $k$-group (not necessarily connected).
Let $G_A$ denote its identity component, it is a closed, connected, normal subgroup of finite index in $G$.
Since the field $k$ is of characteristic zero, the Lie algebra $\Der (A)$ is isomorphic to
$\Lie (G)=\Lie (G_A)$. The group of inner automorphisms $\Inner (A)$ is a closed, connected, normal
subgroup of $G$, so that the group of outer automorphisms $G/\Inner (A)$ is well
defined and also acquires the structure of an affine algebraic $k$-group, and the Lie algebra
$\Lie (G/\Inner (A))$ is isomorphic to $\Out (A)=\Der(A)/\ad(A)$, the Lie algebra of outer derivations of $A$;
see, e.g. \cite[Corollary~13.2]{Hu}, \cite[Proposition~3.1]{Str}.

Recently, this structure attracted considerable attention, see
\cite{CSS}, \cite{ER}, \cite{LRD}, \cite{RDSS} and the references therein.
It is an invariant of the derived equivalence class
of $A$ and is related to the representation type of $A$.

It would be interesting to understand whether one can use the multiplicative and additive $\Sha (A)$ in
this circle of problems. First, one has to answer some basic questions. Recall that we assume $A$ to be a finite-dimensional
associative unital algebra over a field $k$ of characteristic zero.

\begin{lemma}
$\AIAut (A)$ is a normal subgroup of $\Aut (A)$.
\end{lemma}

\begin{proof}
One has to repeat, word for word, the proof of Lemma \ref{lem:norm}.
\end{proof}

\begin{cor}
$\AIAut (A)$ is a normal subgroup of $G_A$. \qed
\end{cor}

\begin{question} \label{q:closed}
Is $\AIAut (A)$ a closed subgroup of $G_A$?
\end{question}

We see no reason to have an affirmative answer for an arbitrary algebra $A$. See, however,
Theorem \ref{th:GW} below.

Clearly, $\Inner (A)$ is a closed, connected, normal subgroup of $\Aut (A)$,
so that if for a certain algebra $A$ Question \ref{q:closed} is answered in the affirmative, then $\Sha_{\text{\rm{m}}}(A)$
becomes a closed subgroup of $\Out (A)$, thus acquiring the structure of an affine algebraic $k$-group.
This gives rise to the following observation. 

\begin{lemma} \label{lem:closed}
Suppose that $\AIAut (A)$ is a closed subgroup of $\Aut (A)$.
Then the Lie algebras $\Lie (\Sha_{\text{\rm{m}}}(A))$ and $\Sha_{\text{\rm{a}}}(A)$
are isomorphic.
\end{lemma}

\begin{proof}
Under the standard correspondence between algebraic groups and Lie algebras,
which takes elements of $G$ (=automorphisms of $A$) to derivations of $A$ as mentioned above
(see \cite[Section~10.7 and Corollary~13.2]{Hu}), inner automorphisms of $A$ go to inner
derivations of $A$ and similarly, almost inner automorphisms of $A$ go to almost inner
derivations of $A$. Since the characteristic of $k$ is zero, this correspondence gives rise
to isomorphisms of Lie algebras $\Lie(\Inner (A))\cong \ad (A)$, $\Lie(\AIAut(A))\cong \AID(A)$
(again, see \cite[Corollary~13.2]{Hu}), hence
$$
\begin{aligned}
\Lie (\Sha_{\text{\rm{m}}}(A)) & = \Lie(\AIAut(A)/\Inner(A))\cong \Lie(\AIAut(A))/\Lie(\Inner(A))
\cong \AID(A)/\ad (A)\\
&=\Sha_{\text{\rm{a}}}(A),
\end{aligned}
$$
which proves the lemma. \end{proof}

Thus, under the assumptions of Lemma \ref{lem:closed}, any eventual example of an algebra $A$ with nonzero $\Sha (A)$,
either additive or multiplicative, would immediately yield a required example for the other structure.

Here is an important special case.

\begin{theorem} \label{th:GW}
Let $k=\mathbb C$. With the notation as above, assume in addition that the algebraic $k$-group $G_A$ is nilpotent. Then
\begin{itemize}
\item[(i)] $\AIAut(A)$ is a closed normal subgroup of $G_A$ with Lie algebra $\AID (A)$;
\item[(ii)] the Lie algebras $\Lie (\Sha_{\text{\rm{m}}}(A))$ and $\Sha_{\text{\rm{a}}}(A)$
are isomorphic.
\end{itemize}
\end{theorem}

\begin{proof}
As in Theorem \ref{th:ideal-lie}, the proof of (i) follows, {\it mutatis mutandis}, the proof of Theorem 2.3 in \cite{GW}.


We obtain (ii) by combining (i) with Lemma \ref{lem:closed}.
\end{proof}

\begin{remark}
So far it is not clear whether there exists $A$ fitting into the frame
of Theorem \ref{th:GW} and providing an example with nonzero $\Sha (A)$.
One can try to produce such an $A$ using the results of R.~D.~Pollack \cite{Po}, particularly
Theorem 1.6 and Example 1.7.
\end{remark}

\section{Concluding parallels} \label{sec:par}

\subsection{Structure properties}
Actually, very little is known on the structure properties of the Tate--Shafarevich groups and
algebras considered above. As of now, main vague parallels arise from looking at $\Sha(G)$ of
finite groups $G$, see \cite{Ku2} for more details. 

Here are some basic questions. Throughout we assume that $\g$ is a finite-dimensional Lie algebra and
$A$ is a finite-dimensional associative unital algebra. As to the characteristic of the ground field,
in this section we shall consider two separate cases, following a suggestion of the referee.

\subsubsection{Case $\ch (k)=0$}

\begin{question}
\begin{itemize}
\item[]
\item[(i)] Does there exist $\g$ such that the algebra $\Sha (\g)$ is non-abelian?
\item[(ii)] Does there exist $A$ such that the algebra $\Sha_{\rm{a}}(A)$ is non-abelian?
\item[(iii)] Does there exist $A$ such that the group $\Sha_{\rm{m}}(A)$ is non-abelian?
\end{itemize}
\end{question}

Recall that Sah \cite{Sah} disproved Burnside's statement \cite{Bur2} and exhibited examples of $p$-groups $G$ with
non-abelian $\Sha (G)$, the smallest among them is a group of order $2^{15}$.

Our working hypothesis is that all these questions are answered in the affirmative.

\begin{question} \label{q}
\begin{itemize}
\item[]
\item[(i)] Does there exist $\g$ such that the algebra $\Sha (\g)$ is non-solvable?
\item[(ii)] Does there exist $A$ such that the algebra $\Sha_{\rm{a}}(A)$ is non-solvable?
\item[(iii)] Does there exist $A$ such that the group $\Sha_{\rm{m}}(A)$ is non-solvable?
\end{itemize}
\end{question}

Here we would rather expect that all Tate--Shafarevich algebras and groups appearing in these
questions are solvable. Note that even in the case of finite groups $G$ only a conditional
statement is available. The proof of the solvability in \cite{Sah} contains a gap noticed
by Murai \cite{Mu} who showed that the validity of this assertion depends on the Alperin--McKay conjecture.

\subsubsection{Case $\ch (k)=p>0$}
In this case, one might expect a very different behaviour of all versions of $\Sha$.
A general reason is that in a situation where an algebra has no outer derivations
in characteristic zero, it may have such in positive characteristic. Here are several
instances of this phenomenon.

$\bullet$ All derivations of the group algebra $k[G]$ of a finite group $G$ are inner
if $\ch (k)=0$ but there are outer derivations in the modular case, see, e.g. \cite{ArKo}.

$\bullet$ All derivations of a simple finite-dimensional Lie $k$-algebra are inner
(first Whitehead's lemma) but there are outer derivations in the modular case, even
for classical Lie algebras, where such algebras are classified for $p>3$, see \cite[Lemma~2.7]{BGP}.
There are more examples of the same flavour for non-classical simple algebras, as well as
for classical algebras in characteristics 2 and 3. The latter cases provide algebras $\g$
with non-solvable $\Out(\g)$ thus refuting a conjecture of Zassenhaus. See \cite{BM}
for details and a more recent paper \cite{BMPG} for an infinite family of counter-examples
in characteristic 3.

It is not clear what are the Tate--Shafarevich algebras in all these exceptional cases.
In light of the aforementioned counter-examples to the Zassenhaus conjecture, it would
be particular tempting to answer Question \ref{q}(i) in this set-up.
This is an interesting topic for future research, and we thank the referee for bringing
the aforementioned examples to our attention.

\subsection{Eventual generalizations} It is tempting to extend the notions introduced in this paper
to other algebraic structures for which there exists a developed cohomology theory, with a goal to
define, explore and apply analogues of Tate--Shafarevich sets to relevant problems of
different categories. One has to try to equip these sets, if possible, with an additional structure (group
or algebra). Also, it is very desirable to include the structure under consideration in a relevant triad,
if such exists, similarly to the classical triad consisting of Lie algebras, associative algebras and groups.

A prototypical example where one can observe these structures is the set of $(n\times n)$-matrices
giving rise to the Lie algebra $\g=\mathfrak{gl}_n(k)$ if equipped by the bracket $(X,Y)\mapsto XY-YX$, to the
associative algebra $A={\mathrm M}_n(k)$ if equipped by matrix multiplication, and to the group $G={\mathrm{GL}}_n(k)$
of the invertible elements of the latter algebra. One has to emphasize, however, that each of the three objects has
trivial Tate--Shafarevich set. Indeed, the Skolem--Noether theorem implies this for $\Sha_{\textrm{a}}(A)$ and $\Sha_{\textrm{m}}(A)$
because of the absence of outer automorphisms and derivations, $\Sha(\g)=0$ by first Whitehead's lemma, and $\Sha(G)=1$
by a theorem of Hideo Wada \cite{Wa}. To complete the picture, one can mention a theorem of Feit and Seitz
\cite{FS} stating that $\Sha(S)=1$ for any finite simple nonabelian group $S$. Thus, looking for eventual analogies, one
has to leave the realm of simple (or, more generally, semisimple) objects in favour of the study of nilpotent ones.

Below we list some possible situations where the generalizations we are looking for seem reachable.

\medskip

$\bullet$ {\it Malcev algebras}

Malcev algebras arise from Lie algebras when one relaxes the Jacobi identity replacing it with a weaker condition,
and keeps the anti-commutativity, see, e.g. \cite{Sag}, \cite{KS}.
One can start with derivations of such an algebra $M$, where inner derivations are defined as in \cite{Sch},
and introduce almost inner derivations.
The arising set $\Sha (M)$ carries a structure of vector space but not necessarily a structure of Lie algebra.
The relevant triad to be considered should include alternative algebras (as a substitute for associative algebras) and
Moufang loops (as a substitute for groups). Note that analogues of Lie theorems in this set-up are available, see the paper
of Kerdman \cite{Ke} and the references therein. As in the classical case, the `Lie correspondence' between Moofang loops
and Malcev algebras works particularly well in the nilpotent case, see \cite{GRSS}.

\medskip

$\bullet$ {\it Leibniz algebras}

Leibniz algebras arise from Lie algebras in an opposite way, when one keeps the Jacobi identity and drops
the anti-commutativity condition, see, e.g. \cite{Lo1}, \cite{AOR}. Here there is a well-developed (co)homology theory
\cite{LP}, \cite{Pi}, and the Leibniz adjoint cohomology $HL^1(L,L)$ of a Leibniz algebra $L$ is the
space of outer derivations of $L$, see \cite{LP}. One then can introduce almost inner derivations and $\Sha (L)$ as in the case of Lie algebras,
see \cite{AdKu} where the authors provide examples of $L$ with nonzero $\Sha (L)$.
The eventual triad should include dialgebras \cite{Lo3} (as a substitute for associative algebras)
and so-called `coquecigrues' \cite{Lo2}, \cite{JP}, whose existence is known for several classes of Leibniz algebras
and an analogue of Lie theory is established. Hopefully, $\Sha (L)$ may reveal some related geometric phenomena.

\medskip

$\bullet$ {\it Poisson algebras}

Recall that a Poisson algebra $A$ is equipped with structures of associative algebra and Lie algebra which are
related by the Leibniz identity. The Poisson adjoint cohomology $H^1_{\pi}(A)$ is the quotient
$\Der_{\pi}(A)/\Ham(A)$, where $\Der_{\pi}(A)$ is the Lie algebra of Poisson derivations (i.e. derivations of both
associative and Lie structures) and $\Ham (A)$ is the ideal of Hamiltonian derivations. 
As in the preceding cases, we can introduce almost inner derivations and define $\Sha (A)$.
Here one can hope to use the Duflo isomorphism \cite{PT} for establishing connections and analogies
with other versions of $\Sha$. We hope that this object admits a conceptual interpretation within the frame
of Poisson geometry, in the spirit of \cite{Wei}. But this is another story.

\medskip

\noindent{\it Acknowledgements.} We thank Be'eri Greenfeld, Alexei Kanel-Belov, Leonid Makar-Limanov and Pradeep Kumar Rai for useful
discussions on various aspects of this work. We are grateful to the anonymous referee for careful reading and thoughtful remarks.

\enddocument